\documentclass[10pt,reqno]{amsart}
\usepackage{amsmath} 
\usepackage{amssymb}
\usepackage{dsfont}
\usepackage[dvips,draft,final]{graphics}
\usepackage{amssymb,amsmath}
\usepackage[T1]{fontenc}
\usepackage{fancyhdr}
\usepackage{url}

\usepackage{color}
\usepackage{graphicx}

\def\squarebox#1{\hbox to #1{\hfill\vbox to #1{\vfill}}}

\theoremstyle{plain}

\newtheorem{Thm}{Theorem}

\newtheorem{lem}{Lemma}

\newcommand{\bel}{\begin{equation} \label}
\newcommand{\ee}{\end{equation}}
\newcommand{\re}{\mathfrak R}
\newcommand{\im}{\mathfrak I}

\newcommand{\R}{\mathbb{R}}

\def\epsilon{\varepsilon}
\def\phi {\varphi}

\newtheorem{rem}{Remark}
\newtheorem{prop}{Proposition}

\providecommand{\abs}[1]{\left\lvert#1\right\rvert}
\providecommand{\norm}[1]{\left\lVert#1\right\rVert}
\numberwithin{equation}{section}

\renewcommand{\leq}{\leqslant}
\renewcommand{\geq}{\geqslant}
\providecommand{\abs}[1]{\left\lvert#1\right\rvert}
\providecommand{\norm}[1]{\left\lVert#1\right\rVert}

\def\beq{\begin{equation}}
\def\eeq{\end{equation}}
\newcommand{\bea}{\begin{eqnarray}}
\newcommand{\eea}{\end{eqnarray}}
\newcommand{\beas}{\begin{eqnarray*}}
\newcommand{\eeas}{\end{eqnarray*}}

{

\begin{document}

\title[Determination of a time-dependent coefficient for  wave equations]{Unique determination of a time-dependent potential for  wave equations from partial   data}

\author[Yavar Kian]{Yavar Kian}
\maketitle
\begin{center} \footnotesize{ CPT, UMR CNRS 7332,\\ Aix Marseille Universit\'e,\\ 13288 Marseille, France,\\ and Universit\'e de Toulon,\\ 83957 La Garde, France\\ yavar.kian@univ-amu.fr}\end{center}

\begin{abstract}
We consider   the inverse problem of  determining   a time-dependent potential $q$, appearing in   the wave equation $\partial_t^2u-\Delta_x u+q(t,x)u=0$ in $Q=(0,T)\times\Omega$ with $\Omega$  a  $C^2$ bounded domain of $\R^n$, $n\geq2$, from  partial observations of the solutions on $\partial Q$.  We prove global unique determination  of a coefficient $q\in L^\infty(Q)$. \\


\end{abstract}

\section{Introduction}
\label{sec-intro}
\setcounter{equation}{0}
\subsection{Statement of the problem }
We fix $\Omega$ a $\mathcal C^2$ bounded domain  of $\R^n$, $n\geq2$, and we set  $\Sigma=(0,T)\times\partial\Omega$, $Q=(0,T)\times\Omega$ with $0<T<\infty$.  We consider the wave equation \begin{equation}\label{wave}\partial_t^2u-\Delta_x u+q(t,x)u=0,\quad (t,x)\in Q,\end{equation}
where the  potential $q\in L^\infty(Q)$ is assumed to be real valued.  We study the inverse problem of determining $q$ from observations  of  solutions of \eqref{wave} on $\partial Q$. 

It is well known that for $T>\textrm{Diam}(\Omega)$ the data
\begin{equation}\label{data1} \mathcal A_q=\{(u_{|\Sigma},\partial_\nu u_{|\Sigma}):\ u\in L^2(Q),\ \Box u+qu=0,\ u_{|t=0}=\partial_tu_{|t=0}=0\}\end{equation}
determines uniquely a time-independent potential $q$ (e.g. \cite{RS1}). Here $\nu$ denotes the outward unit normal vector to $\Omega$ and from now on $\Box$ denotes the differential operator $\partial_t^2-\Delta_x$. It has been even proved that  partial knowledge of $\mathcal A_q$ determines  a time-independent potential $q$ (e.g. \cite{E1}). In contrast to time-independent potentials, we can not recover  the restriction of a general time dependent potential $q$  on the set $$D=\{(t,x)\in Q:\  0<t<\textrm{Diam}(\Omega)/2,\  \textrm{dist}(x,\partial\Omega)< t\}$$ from the data $\mathcal A_q$. Indeed, assume that $\Omega=\{x\in\R^n:\ |x|<R\}$, $T>R>0$. Now let $u$ solve
\[\Box u=0,\ u_{|\Sigma}=f,\quad u_{|t=0}=\partial_tu_{|t=0}=0.\]
 with $f\in H^1(\Sigma)$ satisfying $f_{|t=0}=0$. Since $u_{|t=0}=\partial_tu_{|t=0}=0$,  the finite speed of propagation implies that $u_{|D}=0$. Therefore, for any $q\in\mathcal C^\infty_0(D)$,  we have $qu=0$ and $u$ solves
\[\Box u+qu=0,\ u_{|\Sigma}=f,\quad u_{|t=0}=\partial_tu_{|t=0}=0.\]
This last result implies that for any $q\in\mathcal C^\infty_0(D)$ we have $\mathcal A_q=\mathcal A_0$ where $\mathcal A_0$ stands for $\mathcal A_q$ when $q=0$. 

 Facing this obstruction to uniqueness, it appears that  four  different approaches have been considered so far to solve this problem:\\
1) Considering the equation \eqref{wave} for any time $t\in\R$ instead of $0<t<T$ (e.g. \cite{RS}, \cite{S}).\\
2) Recovering  the restriction on a subset of $Q$ of a time-dependent potential $q$  from the data $\mathcal A_q$ (e.g. \cite{RR}).\\
3) Recovering  a time-dependent potential $q$ from the extended data $C_q$ (e.g. \cite{I}) given by
\[C_q=\{(u_{\vert\Sigma},u_{\vert t=0}, \partial_tu_{\vert t=0},\partial_\nu u_{\vert\Sigma},u_{\vert t=T},\partial_tu_{\vert t=T}):\  u\in L^2(Q),\ (\partial_t^2-\Delta_x +q)u=0\}.\]
4) Recovering  time-dependent coefficients that are analytic with respect to the $t$ variable (e.g. \cite{E2}).\\

Therefore, it seems that the only results of unique global determination of a time-dependent potential $q$ proved so far (at finite time) involve strong smoothness assumptions such as analyticity with respect to the $t$ variable or the important set of data $C_q$. In the present paper we investigate some general conditions that guaranty unique determination  of general time-dependent potentials without involving an important set of data.  More precisely, our goal is to prove unique global determination of a general time-dependent potential $q$ from partial knowledge of  the   set of data $C_q$.  
\subsection{Physical and mathematical interest }

 Physically speaking, our inverse problem can be stated as the determination of physical properties such as the time evolving density of an inhomogeneous medium by probing it with disturbances generated on some parts  of the boundary and at initial time. The data is the response of the medium to these disturbances, measured on some parts of the boundary and at the end of the experiment, and the purpose is to recover the function $q$ which measures the property of the medium. Note also that the determination of time dependent potentials can be associated to models where it is necessary to take into account the evolution in time of the perturbation.

We also precise that the determination of time-dependent potentials can  be an important tool  for the more difficult problem of determining a non-linear term appearing in a nonlinear wave equation from observations of the solutions in $\partial Q$. Indeed, in \cite{I2} Isakov applied such results for the determination of a semilinear term appearing in a semilinear parabolic equation from observations of the solutions in $\partial Q$.
\subsection{Existing papers }

In recent years the determination of  coefficients for hyperbolic equations  from boundary measurements  has  been growing in interest.  Many authors have considered this problem with an observation given by the set $\mathcal A_q$ (see \eqref{data1}). In \cite{RS1}, Rakesh and  Symes proved that  $\mathcal A_q$ determines uniquely a time-independent  potential $q$  and   \cite{I1} proved unique determination of a potential and a damping coefficient.  The uniqueness by partial boundary observations has been considered in \cite{E1}. For sake of completeness we also mention that the  stability issue related to this problem has been treated by \cite{BJY,IS,Ki,Mo,SU,SU2}. Note that \cite{Ki} extended the results of \cite{RS1} to  time-independent coefficients of order zero in an unbounded cylindrical domain. It has been proved that  measurements on a bounded subset determine some classes of coefficients including periodic coefficients and compactly supported coefficients.

All the above mentioned results are concerned  with time-independent coefficients. Several authors considered the problem of determining time-dependent coefficients for hyperbolic equations. In \cite{St}, Stefanov proved unique determination of a time-dependent potential for the wave equation  from the knowledge of scattering data which is equivalent to the problem with boundary measurements.   In \cite{RS}, Ramm and  Sj\"ostrand considered the determination of a time-dependent potential $q$ from the data $(u_{|\R\times\partial\Omega}, \partial_\nu u_{|\R\times\partial\Omega})$ of forward solutions of \eqref{wave} on the infinite time-space cylindrical domain $\R_t\times\Omega$  instead of $Q$ ($t\in\R$ instead of $0<t<T<\infty$).   Rakesh and  Ramm \cite{RR} considered the  problem at finite time on $Q$, with $T>\textrm{Diam} (\Omega)$, and they determined uniquely $q$ restricted to some  subset of $Q$ from $\mathcal A_q$.  Isakov established in \cite[Theorem 4.2]{I} unique determination of  general time-dependent potentials on the whole domain $Q$ from the extended data $C_q$.  Applying  a result of unique continuation borrowed from \cite{T}, Eskin \cite{E2} proved that the data $\mathcal A_q$  determines time-dependent coefficients  analytic with respect to the time variable $t$. Salazar \cite{S} extended the result of \cite{RS} to more general coefficients. Finally, \cite{W} stated stability in the recovery of  X-ray transforms of time-dependent potentials on a  manifold and \cite{A} proved log-type stability in the determination of time-dependent potentials from the data considered by \cite{RR} and \cite{I}.

We also mention that \cite{Ch,CK,CKS,GK} examined the determination of time-dependent coefficients for parabolic and Schr\"odinger equations and  proved stability estimate for these problems.

\subsection{Main result}
In order to state our main result, we  first introduce some intermediate tools and notations. 
For all $\omega\in\mathbb S^{n-1}=\{y\in\R^n:\ \abs{y}=1\}$ we introduce the $\omega$-shadowed and $\omega$-illuminated faces
\[\partial\Omega_{+,\omega}=\{x\in\partial\Omega:\ \nu(x)\cdot\omega\geq0\},\quad \partial\Omega_{-,\omega}=\{x\in\partial\Omega:\ \nu(x)\cdot\omega\leq0\}\]
of $\partial\Omega$. Here, for all $k\in\mathbb N^*$, $\cdot$ denotes the scalar product in $\R^k$ defined by
\[ x\cdot y=x_1y_1+\ldots +x_ky_k,\quad x=(x_1,\ldots,x_k)\in \R^k,\ y=(y_1,\ldots,y_k)\in \R^k.\]
We associate to $\partial\Omega_{\pm,\omega}$ the part of the lateral boundary $\Sigma$ given by $\Sigma_{\pm,\omega}=(0,T)\times\partial\Omega_{\pm,\omega}$. From now on we fix $\omega_0\in \mathbb S^{n-1}$ and we consider $F=(0,T)\times F'$ (resp $G=(0,T)\times G'$) with $F'$ (resp $G'$) an open  neighborhood of $\partial\Omega_{+,\omega_0}$ (resp $\partial\Omega_{-,\omega_0}$) in $\partial\Omega$.

 The main purpose of this paper is to prove  the unique global  determination of a time-dependent and real valued  potential $q\in L^\infty(Q)$ from the data
\[C_q^*=\{(u_{|\Sigma},\partial_t u_{|t=0},\partial_\nu u_{|G}, u_{|t=T}):\ u\in L^2(Q),\ \Box u+qu=0,\ u_{|t=0}=0,\ \textrm{supp}u_{|\Sigma}\subset F\}.\]
See also Section 2 for a rigorous definition of this set. Our main result can be  stated as follows.

\begin{Thm}\label{thm1} 
Let $q_1,\ q_2 \in L^\infty(Q)$ . Assume that 
\begin{equation}\label{thm1a} C_{q_1}^*=C_{q_2}^*.\end{equation}
Then $q_1=q_2$.
\end{Thm}

Note that our uniqueness result is stated for bounded potentials with, roughly speaking, half of the data $C_q$ considered in \cite[Theorem 4.2]{I} which seems to be, with \cite{A}, the only result of unique global determination of general time-dependent coefficients for the wave equation, at finite time, in the mathematical literature. More precisely, we consider  $u\in L^2(Q)$ solutions of $(\partial_t^2-\Delta +q)u=0$, on $Q$, with initial condition  $u_{\vert t=0}=0$ and  Dirichlet boundary condition  $u_{\vert\Sigma}$ supported on $F$ (which, roughly speaking, corresponds  to half of the boundary). Moreover, we exclude the data $\partial_t u_{\vert t=T}$ and  we consider the Neumann data $\partial_\nu u$ only on $G$ (which, roughly speaking, corresponds to the other half of the boundary).
We also mention that in contrast to \cite{E2}, we do not use results of unique continuation where the analyticity of the coefficients with respect to $t$ is required.
To our best knowledge condition \eqref{thm1a} is the weakest condition that guaranties global uniqueness of general time dependent potentials. Moreover, taking into account the obstruction to uniqueness given by domain of dependence arguments (see Subsection 1.1), the restriction to solutions $u$ of \eqref{wave} satisfying $u_{|t=0}=0$ seems close to the best condition that we can expect on the initial  data for the determination of time-dependent potentials.

The main tools in our analysis are   geometric optics (GO in short) solutions and  Carleman estimates. Following an approach used for elliptic equations (e.g. \cite{BU,KSU,NS}) and for determination of time-independent potentials by \cite{BJY}, we construct two kind of GO solutions: GO solutions lying in $H^1(Q)$
without condition on  $\partial Q$ (see Section 3) and GO solutions associated to \eqref{wave} that vanish on parts of  $\partial Q$ (see Section 5). With these solutions and some Carleman estimates with linear weight (see Section 4), we prove Theorem \ref{thm1}.

\subsection{Outline}

This paper is organized as follows. In Section 2 we  give a suitable definition of the set of data $C_q^*$ and  we define the associated boundary operator. In Section 3, using some results of \cite{Ch} and \cite{Ho2}, we build suitable GO solutions associated to \eqref{wave} without conditions on  $\partial Q$. In Section 4, we establish a Carleman estimate for the wave equation with linear weight. In Section 5, we use the Carleman estimate introduced in Section 4 to   build  GO solutions associated to \eqref{wave} that vanish on parts of  $\partial Q$. More precisely, we build GO  $u$ which are solutions of \eqref{wave} with $u_{|t=0}=0$ and supp$u_{|\Sigma}\subset F$. In Section 6 we combine all the results of the previous sections in order to prove Theorem \ref{thm1}. We prove also some auxiliary results in the appendix.
\vspace{5mm}
\ \\
\textbf{Acknowledgements}. The author would like to thank Mourad Bellassoued, Mourad Choulli and Eric Soccorsi   for their   remarks and  suggestions.
\section{Preliminary results}
The goal of this section is to give a suitable definition to the set of data $C_q^*$ and to introduce some properties of the solutions of \eqref{wave} for any $q\in L^\infty(Q)$ real valued. We first introduce  the  space
\[J=\{u\in L^2(Q):\ (\partial_t^2-\Delta_x)u=0\}\]
and topologize it as a closed subset of $L^2(Q)$.  We work with the space
\[H_{\Box}(Q)=\{u\in L^2(Q):\ \Box u=(\partial_t^2-\Delta_x) u\in L^2(Q)\},\]
 with the norm
\[\norm{u}^2_{H_{\Box}(Q)}=\norm{u}_{L^2(Q)}^2+\norm{(\partial_t^2-\Delta_x) u}_{L^2(Q)}^2.\]
Repeating some arguments of  \cite[Theorem 6.4, Chapter 2]{LM1} we prove in the appendix (see Theorem \ref{density}) that $H_\Box(Q)$ embedded continuously into the closure of $\mathcal C^\infty(\overline{Q})$ in the space
\[K_\Box(Q)=\{u\in H^{-1}(0,T;L^2(\Omega)):\ \Box u=(\partial_t^2-\Delta_x) u\in L^2(Q)\}\]
topologized by the norm
\[\norm{u}^2_{K_{\Box}(Q)}=\norm{u}_{H^{-1}(0,T;L^2(\Omega))}^2+\norm{(\partial_t^2-\Delta_x) u}_{L^2(Q)}^2.\]
 Then, following   \cite[Theorem 6.5, Chapter 2]{LM1}, we prove in the appendix that the maps
\[\tau_0w=(w_{\vert\Sigma},w_{\vert t=0},\partial_t w_{\vert t=0}) ,\quad \tau_1w=(\partial_\nu w_{\vert\Sigma},w_{\vert t=T},\partial_t w_{\vert t=T}), \quad w\in \mathcal C^\infty(\overline{Q}),\]
can be extended continuously to $\tau_0:H_{\Box}(Q)\rightarrow H^{-3}(0,T; H^{-\frac{1}{2}}(\partial\Omega))\times H^{-2}(\Omega)\times H^{-4}(\Omega)$, $\tau_1:H_{\Box}(Q)\rightarrow H^{-3}(0,T; H^{-\frac{3}{2}}(\partial\Omega))\times H^{-2}(\Omega)\times H^{-4}(\Omega)$ (see Proposition \ref{trace}). Here for all $ w\in \mathcal C^\infty(\overline{Q})$ we set
\[\tau_0w=(\tau_{0,1}w,\tau_{0,2}w,\tau_{0,3}w),\quad \tau_1w=(\tau_{1,1}w,\tau_{1,2}w,\tau_{1,3}w),\]
where
\[ \tau_{0,1}w=w_{\vert\Sigma},\    \tau_{0,2}w=w_{\vert t=0},\ \tau_{0,3}w=\partial_tw_{\vert t=0},\  \tau_{1,1}w=\partial_\nu w_{\vert\Sigma},\ \tau_{1,2}w=w_{\vert t=T},\ \tau_{1,3}w=\partial_tw_{\vert t=T}.\]
Therefore, we can introduce
\[\mathcal H(\partial Q)=\{\tau_0u:\ u\in H_{\Box}(Q)\}\subset  H^{-3}(0,T; H^{-\frac{1}{2}}(\partial\Omega))\times H^{-2}(\Omega)\times H^{-4}(\Omega).\]
Following \cite{BU} and \cite{NS}, in order to define an appropriate topology on $\mathcal H(\partial Q)$ we consider the restriction of $\tau_0$ to the space $J$.
\begin{prop}\label{p5} The restriction of $\tau_0$ to $J$ is one to one and onto.
\end{prop}
 \begin{proof} Let $v_1,v_2\in J$ with $\tau_0v_1=\tau_0v_2$. Then, in light of  the theory introduced in \cite[Section 8, Chapter 3]{LM1}, there exists $F\in H_\Box (Q)$ such that, for $j=1,2$, we have $v_j=F+w_j$ with $w_j\in\mathcal C^1([0,T]; L^2(\Omega))\cap \mathcal C([0,T]; H^1_0(\Omega))$ solving

\[\left\{ \begin{array}{rcll} \partial_t^2w_j-\Delta_x w_j& = & -\Box F, & (t,x) \in Q ,\\ 

{w_j}_{|t=0}=\partial_t {w_j}_{\vert t=0}&=&0,&\\
   {w_j}_{\vert\Sigma}& = & 0.& \end{array}\right.\]
Then, the uniqueness of solutions of this  initial boundary value problem (IBVP in short) implies that $v_1=v_2$. Thus, the restriction of $\tau_0$ to $J$ is one to one. Now let $(g,v_0,v_1)\in \mathcal H(\partial Q)$. There exists $S\in H_{\Box}(Q)$ such that $\tau_0 S=(g,v_0,v_1)$. Consider the  initial boundary value problem
\[\left\{ \begin{array}{rcll} \partial_t^2v-\Delta_x v& = & -\Box S, & (t,x) \in Q ,\\ 
v_{\vert t=0}=\partial_t v_{\vert t=0}&=&0,&\\
v_{\vert\Sigma}& = & 0.&\end{array}\right.\]
Since $-\Box S\in L^2(Q)$,  we deduce that this IBVP admits a unique solution $v\in \mathcal C^1([0,T];L^2(\Omega))\cap \mathcal C([0,T];H^1_0(\Omega))$. Then, $u=v+S\in L^2(Q)$ satisfies $(\partial_t^2-\Delta_x) u=0$ and 
$\tau_0u=\tau_0v+\tau_0S=(g,v_0,v_1)$. Thus $\tau_0$ is onto.\end{proof}
From now on, we set $\mathcal P_0$  the inverse of $\tau_0:J\rightarrow \mathcal H(\partial Q)$ and  define the norm of $\mathcal H(\partial Q)$ by
\[\norm{(g,v_0,v_1)}_{\mathcal H(\partial Q)}=\norm{\mathcal P_0(g,v_0,v_1)}_{L^2(Q)},\quad (g,v_0,v_1)\in\mathcal H(\partial Q).\]
In the same way, we introduce the space $\mathcal H_F(\partial Q)$ defined by 
\[\mathcal H_F(\partial Q)=\{(\tau_{0,1}h,\tau_{0,3}h):\ h\in H_\Box(Q),\ \tau_{0,2}h=0,\ \textrm{supp}\tau_{0,1}h\subset F\}\]
with the associated norm given by
\[\norm{(g,v_1)}_{\mathcal H_F(\partial Q)}=\norm{(g,0,v_1)}_{\mathcal H(\partial Q)},\quad (g,v_1)\in \mathcal H_F(\partial Q).\]
One can easily check that the space $\mathcal H_F(\partial Q)$ embedded continuously into $\mathcal H(\partial Q)$.
Let us consider  the IBVP
\begin{equation}\label{eq1}\left\{\begin{array}{ll}\partial_t^2u-\Delta_x u+q(t,x)u=0,\quad &\textrm{in}\ Q,\\  u(0,\cdot)=0,\quad \partial_tu(0,\cdot)=v_1,\quad &\textrm{in}\ \Omega,\\ u=g,\quad &\textrm{on}\ \Sigma.\end{array}\right.\end{equation}
We are now in position to state existence and uniqueness of solutions of this IBVP for $(g,v_1)\in \mathcal H_F(\partial Q)$.
\begin{prop}\label{p6} Let $(g,v_1)\in \mathcal H_F(\partial Q)$ and $q\in L^\infty(Q)$. Then the IBVP \eqref{eq1} admits a unique weak solution $u\in L^2(Q)$ satisfying 
\bel{p6a}
\norm{u}_{L^2(Q)}\leq C\norm{(g,v_1)}_{\mathcal H_F(\partial Q)}
\ee
and the boundary operator $B_q: (g,v_1)\mapsto (\tau_{1,1}u_{|G}, \tau_{1,2}u)$ is a bounded operator from $\mathcal H_F(\partial Q)$ to\\
 $H^{-3}(0,T; H^{-\frac{3}{2}}(G))\times H^{-2}(\Omega)$.
\end{prop}
 \begin{proof} We split $u$ into two terms $u=v+\mathcal P_0(g,0,v_1)$ where $v$ solves
\bel{Eq1}\left\{ \begin{array}{rcll} \partial_t^2v-\Delta_x v+qv& = & -q\mathcal P_0(g,0,v_1), & (t,x) \in Q ,\\ 

v_{\vert t=0}=\partial_t v_{\vert t=0}&=&0,&\\
 v_{\vert\Sigma}& = & 0.& \end{array}\right.
\ee
Since $\mathcal P_0(g,0,v_1)\in L^2(Q)$, the IBVP \eqref{Eq1} admits a unique solution $v\in \mathcal C^1([0,T];L^2(\Omega))\cap \mathcal C([0,T];H^1_0(\Omega))$ (e.g. \cite[Section 8, Chapter 3]{LM1}) satisfying
\bel{p6b}
\norm{v}_{\mathcal C^1([0,T];L^2(\Omega))}+\norm{v}_{\mathcal C([0,T];H^1_0(\Omega))} \leq C\norm{-q\mathcal P_0(g,0,v_1)}_{L^2(Q)}\leq C\norm{q}_{L^\infty(Q)}\norm{\mathcal P_0(g,0,v_1)}_{L^2(Q)}.\ee
Therefore, $u=v+\mathcal P_0(g,0,v_1)$ is the unique solution of \eqref{eq1} and estimate \eqref{p6b} implies \eqref{p6a}. Now let us show the last part of the proposition. For this purpose fix $(g,v_1)\in\mathcal H_F(\partial Q)$ and consider $u$ the solution of \eqref{eq1}. Note first that $u\in L^2(Q)$ and $(\partial_t^2-\Delta_x) u=-qu\in L^2(Q)$. Thus,  $u\in H_{\Box}(Q)$ and  $\tau_{1,1}u\in H^{-3}(0,T; H^{-\frac{3}{2}}(\partial\Omega))$ $\tau_{1,2}u \in H^{-2}(\Omega)$ with
\[\norm{\tau_{1,1}u}^2+\norm{\tau_{1,2}u}^2\leq C^2\norm{u}^2_{H_{\Box}(Q)}=C^2(\norm{u}^2_{L^2(Q)}+\norm{qu}^2_{L^2(Q)})\leq C^2(1+\norm{q}^2_{L^\infty(Q)})\norm{u}_{L^2(Q)}^2.\]
Combining this with \eqref{p6a} we deduce that $B_q$ is a bounded operator from $\mathcal H_F(\partial Q)$ to $H^{-3}(0,T; H^{-\frac{3}{2}}(G))\times H^{-2}(\Omega)$.\end{proof}

From now on we consider the set $C_q^*$ to be the graph of the boundary operator $B_q$ given by
\[C_q^*=\{(g,v_1,B_q(g,v_1)):\ (g,v_1)\in \mathcal H_F(\partial Q)\}.\]

\section{Geometric optics solutions without boundary conditions}

In this section we build    geometric optics solutions   $u\in H^1(Q)$ associated to the equation 
\begin{equation}\label{eqGO1}\partial_t^2u-\Delta_xu+q(t,x)u=0\quad \textrm{on } Q.\end{equation}
 More precisely, for $\lambda>1$,  $\omega\in\mathbb S^{n-1}=\{y\in\R^n:\ |y|=1\}$ and $\xi\in\R^{1+n}$ satisfying $\xi\cdot(1,-\omega)=0$, we consider solutions of the form
\begin{equation}\label{GO1} u(t,x)=e^{-\lambda(t+x\cdot\omega)}(e^{-i\xi\cdot(t,x)}+w(t,x)),\quad (t,x)\in Q.\end{equation}
Here $w\in H^1(Q)$ fulfills
\[\norm{w}_{L^2(Q)}\leq \frac{C}{\lambda}\]
with $C>0$ independent of $\lambda$. For this purpose, for all $s\in\R$ and all $\omega\in\mathbb S^{n-1}$, we consider the operators $P_{s,\omega}$ defined by $P_{s,\omega}=e^{-s(t+x\cdot\omega)}\Box e^{s(t+x\cdot\omega)}$. One can check that
\[P_{s,\omega}=p_{s,\omega}(D_t,D_x)=\Box +2s(\partial_t-\omega\cdot \nabla_x)\]
with $D_t=-i\partial_t$, $D_x=-i\nabla_x$ and $p_{s,\omega}(\mu,\eta)=-\mu^2+\abs{\eta}^2+2si (\mu-\omega\cdot\eta)$, $\mu\in\R$, $\eta\in\R^n$. Applying some  results of \cite{Ch} and \cite{Ho2} about solutions of  PDEs with constant coefficients we obtain the following.

\begin{lem}\label{pp1} For every $\lambda>1$ and $\omega\in \mathbb S^{n-1}$ there exists a bounded operator $E_{\lambda,\omega}:\ L^2(Q)\to L^2(Q)$ such that:
\begin{equation}\label{pp1a}P_{-\lambda,\omega} E_{\lambda,\omega}f=f,\quad f\in L^2(Q),\end{equation}
\begin{equation}\label{pp1b} \norm{E_{\lambda,\omega}}_{\mathcal B(L^2(Q))}\leq C\lambda^{-1},\end{equation}
\begin{equation}\label{pp1c} E_{\lambda,\omega}\in \mathcal B(L^2(Q);H^1(Q))\quad \textrm{and}\quad \norm{E_{\lambda,\omega}}_{\mathcal B(L^2(Q);H^1(Q))}\leq C\end{equation}
with $C>$ depending only on $T$ and $\Omega$.\end{lem}
\begin{proof} In light of \cite[Thorem 2.3]{Ch} (see also \cite[Thorem 10.3.7]{Ho2}), there exists a bounded operator $E_{\lambda,\omega}:\ L^2(Q)\to L^2(Q)$, defined from a fundamental solution associated to $P_{-\lambda,\omega}$ (see Section 10.3 of \cite{Ho2}),  such that \eqref{pp1a} is fulfilled. In addition, for all differential operator  $Q(D_t,D_x)$ with ${Q(\mu,\eta)\over \tilde {p}_{-\lambda,\omega}(\mu,\eta)}$ a bounded function, we have $Q(D_t,D_x)E_{\lambda,\omega}\in\mathcal B(L^2(Q))$ and there exists  a constant $C$ depending only on $\Omega$, $T$ such that
\begin{equation}\label{pp1d}\norm{Q(D_t,D_x)E_{\lambda,\omega}}_{\mathcal B(L^2(Q))}\leq C\sup_{(\mu,\eta)\in\R^{1+n}}{|Q(\mu,\eta)|\over \tilde {p}_{-\lambda,\omega}(\mu,\eta)}\end{equation}
with  $\tilde {p}_{-\lambda,\omega}$  given by
\[\tilde{p}_{-\lambda,\omega}(\mu,\eta)=\left(\sum_{k\in\mathbb N}\sum_{\alpha\in\mathbb N^n}|\partial^k_\mu\partial^\alpha_\eta p_{-\lambda,\omega}(\mu,\eta)|^2\right)^{{1\over2}},\quad \mu\in\R,\ \eta\in\R^n.\]
Note that $\tilde{p}_{-\lambda,\omega}(\mu,\eta)\geq \abs{\im \partial_\mu p_{-\lambda,\omega}(\mu,\eta)}=2\lambda$. Therefore, \eqref{pp1d} implies
\[\norm{E_{\lambda,\omega}}_{\mathcal B(L^2(Q))}\leq C\sup_{(\mu,\eta)\in\R^{1+n}}{1\over \tilde {p}_{-\lambda,\omega}(\mu,\eta)}\leq C\lambda^{-1}\]
and \eqref{pp1b} is fulfilled. In a same way, we have $\tilde{p}_{-\lambda,\omega}(\mu,\eta)\geq \abs{\re \partial_\mu p_{-\lambda,\omega}(\mu,\eta)}=2|\mu|$ and $\tilde{p}_{-\lambda,\omega}(\mu,\eta)\geq \abs{\re \partial_{\eta_i} p_{-\lambda,\omega}(\mu,\eta)}=2|\eta_i|$, $i=1,\ldots,n$ and $\eta=(\eta_1,\ldots,\eta_n)$. Therefore, in view of  \cite[Thorem 2.3]{Ch}, we have $E_{\lambda,\omega}\in \mathcal B(L^2(Q);H^1(Q))$ with
\[\norm{E_{\lambda,\omega}}_{\mathcal B(L^2(Q);H^1(Q))}\leq C\sup_{(\mu,\eta)\in\R^{1+n}}{|\mu|+|\eta_1|+\ldots+|\eta_n|\over \tilde {p}_{-\lambda,\omega}(\mu,\eta)}\leq C(n+1)\]
and \eqref{pp1c} is proved.\end{proof}

Applying this result, we can build  geometric optics solutions of the form \eqref{GO1}.
  \begin{prop}\label{p2} Let $q\in L^\infty(Q)$,  $\omega\in\mathbb S^{n-1}$.  Then, there exists $\lambda_0>1$ such that for $\lambda\geq \lambda_0$ the equation \eqref{eqGO1} admits a solution $u\in H^1(Q)$ of the form \eqref{GO1} with
\begin{equation}\label{p2a}\norm{w}_{H^k(Q)}\leq C\lambda^{k-1},\quad k=0,1,\end{equation}
where $C$ and $\lambda_0$ depend on $\Omega$, $\xi$, $T$, $M\geq \norm{q}_{L^{\infty}(Q)}$.
\end{prop}
\begin{proof} We start by recalling that
\[\begin{aligned}\Box e^{-\lambda(t+x\cdot\omega)}e^{-i\xi\cdot(t,x)}&=e^{-\lambda(t+x\cdot\omega)}\left(\Box e^{-i\xi\cdot(t,x)}+2i\lambda \xi\cdot(1,-\omega)e^{-i\xi\cdot(t,x)}\right)\\
\ &=e^{-\lambda(t+x\cdot\omega)}\Box e^{-i\xi\cdot(t,x)},\quad (t,x)\in Q.\end{aligned}\]
Thus, $w$ should be a solution of
\begin{equation}\label{p2b}\partial_t^2w-\Delta_xw-2\lambda(\partial_t-\omega\cdot\nabla_{x})w=-\left((\Box+q) e^{-i\xi\cdot(t,x)}+qw\right).\end{equation}
Therefore, according to Lemma \ref{pp1}, we can define $w$ as a solution of the equation $$w=-E_{\lambda,\omega}\left((\Box+q) e^{-i\xi\cdot(t,x)}+qw\right),\quad w\in L^2(Q)$$ with $E_{\lambda,\omega}\in\mathcal B(L^2(Q))$ given by Lemma \ref{pp1}.
For this purpose, we will use a standard fixed point argument associated to the map
\[\begin{array}{rccl} \mathcal G: & L^2(Q) & \to & L^2(Q), \\
 \ \\ & F & \mapsto &-E_{\lambda,\omega}\left[(\Box+q) e^{-i\xi\cdot(t,x)}+qF\right]. \end{array}\]
Indeed, in view of \eqref{pp1b}, fixing $M_1>0$ , there  exists $\lambda_0>1$ such that for $\lambda\geq \lambda_0$ the map $\mathcal G$ admits a unique fixed point $w$ in $\{u\in L^2(Q): \norm{u}_{L^2(Q)}\leq M_1\}$. In addition, condition \eqref{pp1b}-\eqref{pp1c} imply that $w\in H^1(Q)$ fulfills \eqref{p2a}. This completes the proof.
\end{proof}

\section{Carleman estimates}

This section is devoted to the proof of  Carleman estimates similar  to \cite{BJY} and \cite{BU}. More precisely, we fix $\omega\in\mathbb S^{n-1}$ and we consider the following estimates.

\begin{Thm}\label{c1}  Let $q\in L^\infty(Q)$ and  $u\in\mathcal C^2(\overline{Q})$.  If $u$ satisfies the condition 
 \begin{equation}\label{ttc1}u_{\vert \Sigma}=0,\quad u_{\vert t=0}=\partial_tu_{\vert t=0}=0\end{equation}
then there exists $\lambda_1>1$ depending only on  $\Omega$, $T$ and $M\geq \norm{q}_{L^\infty(Q)}$ such that the estimate
\begin{equation}\label{c1a}\begin{array}{l}\lambda \int_\Omega e^{-2\lambda( T+\omega\cdot x)}\abs{\partial_tu_{\vert t=T}}^2dx+\lambda\int_{\Sigma_{+,\omega}}e^{-2\lambda(t+\omega\cdot x)}\abs{\partial_\nu u}^2\abs{\omega\cdot\nu(x) } d\sigma(x)dt+\lambda^2\int_Qe^{-2\lambda(t+\omega\cdot x)}\abs{u}^2dxdt\\
\leq C\left(\int_Qe^{-2\lambda(t+\omega\cdot x)}\abs{(\partial_t^2-\Delta_x+q)u}^2dxdt+ \lambda^3\int_\Omega e^{-2\lambda(T+\omega\cdot x)}\abs{u_{\vert t=T}}^2dx+\lambda\int_\Omega e^{-2\lambda(T+\omega\cdot x)}\abs{\nabla_xu_{\vert t=T}}^2dx\right)\\
\ \ \ +C\lambda\int_{\Sigma_{-,\omega}}e^{-2\lambda(t+\omega\cdot x)}\abs{\partial_\nu u}^2\abs{\omega\cdot\nu(x) }d\sigma(x)dt\end{array}\end{equation}
holds true for $\lambda\geq \lambda_1$  with $C$  depending only on  $\Omega$, $T$ and $M\geq \norm{q}_{L^\infty(Q)}$.
If $u$ satisfies the condition
\begin{equation}\label{ttc3}u_{\vert \Sigma}=0,\quad u_{\vert t=T}=\partial_tu_{\vert t=T}=0\end{equation}
then the estimate
\begin{equation}\label{c1b}\begin{array}{l}\lambda\int_\Omega e^{2\lambda\omega\cdot x}\abs{\partial_tu_{\vert t=0}}^2dx+\lambda\int_{\Sigma_{-,\omega}}e^{2\lambda(t+\omega\cdot x)}\abs{\partial_\nu u}^2\abs{\omega\cdot\nu(x) } d\sigma(x)dt+\lambda^2\int_Qe^{2\lambda(t+\omega\cdot x)}\abs{u}^2dxdt\\
\leq C\left(\int_Qe^{2\lambda(t+\omega\cdot x)}\abs{(\partial_t^2-\Delta_x+q)u}^2dxdt+ \lambda^3\int_\Omega e^{2\lambda\omega\cdot x}\abs{u_{\vert t=0}}^2dx+\lambda\int_\Omega e^{2\lambda\omega\cdot x}\abs{\nabla_xu_{\vert t=0}}^2dx\right)\\
\ \ \ +C\lambda\int_{\Sigma_{+,\omega}}e^{2\lambda(t+\omega\cdot x)}\abs{\partial_\nu u}^2\abs{\omega\cdot\nu(x) }d\sigma(x)dt\end{array}\end{equation}
holds true for $\lambda\geq \lambda_1$.

\end{Thm}

In order to prove these estimates, we fix $u\in\mathcal C^2(\overline{Q})$ satisfying \eqref{ttc1} (resp \eqref{ttc3}) and we set $v=e^{-\lambda(t+\omega\cdot x)}u$ (resp $v=e^{\lambda(t+\omega\cdot x)}u$)  in such a way that
\begin{equation}\label{c1c}e^{-\lambda(t+\omega\cdot x)}\Box u=P_{\lambda,\omega}v,\quad \left(\textrm{resp }e^{\lambda(t+\omega\cdot x)}\Box u=P_{-\lambda,\omega}v\right).\end{equation}

Then, we consider  the following estimates associated to the weighted operators $P_{\pm\lambda,\omega}$.

\begin{lem}\label{tc} Let $v\in \mathcal C^2(\overline{Q})$ and $\lambda>1$. If $v$ satisfies the condition
 \begin{equation}\label{tc1}v_{\vert \Sigma}=0,\quad v_{\vert t=0}=\partial_tv_{\vert t=0}=0\end{equation}
then the estimate
 \begin{equation}\label{tc2}\begin{array}{l}\lambda\int_\Omega\abs{\partial_tv_{\vert t=T}}^2dx+2\lambda\int_{\Sigma_{+,\omega}}\abs{\partial_\nu v}^2\omega\cdot\nu(x)  d\sigma(x)dt+c\lambda^2\int_Q\abs{v}^2dxdt\\
\leq \int_Q\abs{P_{\lambda,\omega}v}^2dxdt+ 14\lambda \int_\Omega\abs{\nabla_xv_{\vert t=T}}^2dx+2\lambda\int_{\Sigma_{-,\omega}}\abs{\partial_\nu v}^2\abs{\omega\cdot\nu(x) }d\sigma(x)dt\end{array}\end{equation}
holds true for  $c>0$  depending only on  $\Omega$ and $T$. If $v$ satisfies the condition
 \begin{equation}\label{tc3}v_{\vert \Sigma}=0,\quad v_{\vert t=T}=\partial_tv_{\vert t=T}=0\end{equation}
then the estimate
 \begin{equation}\label{tc4}\begin{array}{l}\lambda\int_\Omega\abs{\partial_tv_{\vert t=0}}^2dx+2\lambda\int_{\Sigma_{-,\omega}}\abs{\partial_\nu v}^2\abs{\omega\cdot\nu(x) } d\sigma(x)dt+c\lambda^2\int_Q\abs{v}^2dxdt\\
\leq \int_Q\abs{P_{-\lambda,\omega}v}^2dxdt+ 14\lambda\int_\Omega\abs{\nabla_xv_{\vert t=0}}^2dx+2\lambda\int_{\Sigma_{+,\omega}}\abs{\partial_\nu v}^2\omega\cdot\nu(x) d\sigma(x)dt\end{array}\end{equation}
holds true.\end{lem}

\begin{proof} We start with \eqref{tc2}. For this purpose we fix $v\in \mathcal C^2(\overline{Q})$ satisfying \eqref{tc1} and we consider
\[I_{\lambda,\omega}=\int_Q |P_{\lambda,\omega}v|^2dt dx.\]
Without lost of generality we assume that $v$ is real valued. Repeating some arguments of \cite{BJY} (see the formula 2 line before  (2.4) in page 1225 of  \cite{BJY} and formula (2.5) in page 1226 of  \cite{BJY}) we obtain the following
\[\begin{aligned} I_{\lambda,\omega}\geq \int_Q|\Box v|^2dtdx+c\lambda^2\int_Q\abs{v}^2dxdt+2\lambda\int_{\Sigma}\abs{\partial_\nu v}^2\omega\cdot\nu(x)  d\sigma(x)dt\\+2\lambda\int_\Omega\abs{\partial_tv_{\vert t=T}}^2dx+ 2\lambda \int_\Omega\abs{\nabla_xv_{\vert t=T}}^2dx-4\lambda \int_\Omega(\partial_tv_{\vert t=T})(\omega\cdot\nabla_xv_{\vert t=T})dx.\end{aligned}\]
On the other hand, an application of the Cauchy-Schwarz inequality yields
\[4\lambda \abs{\int_\Omega(\partial_tv_{\vert t=T})(\omega\cdot\nabla_xv_{\vert t=T})dx}\leq {\lambda\over4}\int_\Omega\abs{\partial_tv_{\vert t=T}}^2dx+16\lambda\int_\Omega\abs{\nabla_xv_{\vert t=T}}^2dx\]
and we deduce that
\[\begin{array}{l} I_{\lambda,\omega}+14\lambda \int_\Omega\abs{\nabla_xv_{\vert t=T}}^2dx\\
\geq\int_Q|\Box v|^2dtdx+c\lambda^2\int_Q\abs{v}^2dxdt+2\lambda\int_{\Sigma}\abs{\partial_\nu v}^2\omega\cdot\nu(x)  d\sigma(x)dt+\lambda\int_\Omega\abs{\partial_tv_{\vert t=T}}^2dx.\end{array}\]
From this last estimate we deduce easily \eqref{tc2}. Now let us consider \eqref{tc4}. For this purpose note that for $v$ satisfying \eqref{tc3}, $w$ defined by $w(t,x)=v(T-t,x)$ satisfies \eqref{tc1}. Thus, applying \eqref{tc2} to $w$ with $\omega$ replaced by $-\omega$ we obtain \eqref{tc4}.\end{proof}

In light of Lemma \ref{tc}, we are now in position to prove Theorem \ref{c1}.
\ \\
\textbf{Proof of Theorem \ref{c1}.} Let us first consider the case $q=0$. Note  that for $u$ satisfying \eqref{ttc1}, $v=e^{-\lambda(t+\omega\cdot x)}u$ satisfies  \eqref{tc1}. Moreover, we have \eqref{c1c}
and  \eqref{ttc1} implies $\partial_\nu v_{\vert\Sigma}=e^{-\lambda(t+\omega\cdot x)}\partial_\nu u_{\vert\Sigma}$. Finally, using the fact that
\[\partial_tu=\partial_t(e^{\lambda(t+\omega\cdot x)} v)=\lambda u+e^{\lambda(t+\omega\cdot x)} \partial_tv,\quad \nabla_x v=e^{-\lambda(t+\omega\cdot x)}(\nabla_x u-\lambda u\omega),\]
we obtain
\[\int_\Omega e^{-2\lambda(T+\omega\cdot x)}\abs{\partial_tu_{\vert t=T}}^2dx\leq 2\int_\Omega \abs{\partial_tv_{\vert t=T}}^2dx+2\lambda^2\int_\Omega e^{-2\lambda(T+\omega\cdot x)}\abs{u_{\vert t=T}}^2dx,\]
\[\int_\Omega \abs{\nabla_x v_{\vert t=T}}^2dx\leq 2\lambda^2\int_\Omega e^{-2\lambda(T+\omega\cdot x)}\abs{u_{\vert t=T}}^2dx+2\int_\Omega e^{-2\lambda(T+\omega\cdot x)}\abs{\nabla_x u_{\vert t=T}}^2dx.\]
 Thus, applying the Carleman estimate \eqref{tc2} to $v$, we deduce \eqref{c1a}. For $q\neq0$,
we have 
 \[\abs{\partial_t^2u-\Delta_xu }^2=\abs{\partial_t^2u-\Delta_xu+qu-qu}^2\leq 2\abs{(\partial_t^2-\Delta_x+q)u}^2+2\norm{q}^2_{L^\infty(Q)}\abs{u}^2\]
 and hence if we choose $\lambda_1>2C\norm{q}^2_{L^\infty(Q)}$, replacing $C$ by
 \[C_1=\frac{C\lambda_1^2}{\lambda_1^2-2C\norm{q}^2_{L^\infty(Q)}},\]
 we deduce \eqref{c1a}  from the same estimate when $q=0$. Using similar arguments, we prove \eqref{c1b}.\qed

\begin{rem}\label{rr} Note that, by density, estimate \eqref{c1a} can be extended to any function $u\in\mathcal C^1([0,T]; L^2(\Omega))\cap \mathcal C([0,T]; H^1(\Omega))$ satisfying \eqref{tc1}, $(\partial_t^2-\Delta_x)u\in L^2(Q)$ and $\partial_\nu u\in L^2(\Sigma)$.\end{rem}

\section{Geometric optics solutions vanishing on parts of the boundary}
In this section we fix $q\in L^\infty(Q)$. From now on, for all $y\in\mathbb S^{n-1}$ and all  $r>0$, we set
\[\partial\Omega_{+,r,y}=\{x\in\partial\Omega:\ \nu(x)\cdot y>r\},\quad\partial\Omega_{-,r,y}=\{x\in\partial\Omega:\ \nu(x)\cdot y\leq r\}\]
and $\Sigma_{\pm,r,y}=(0,T)\times \partial\Omega_{\pm,r,y}$. Here and in the remaining of this text we always assume, without mentioning it, that $y$ and $r$ are chosen in such way that $\partial\Omega_{\pm,r,\pm y}$ contain  a non-empty relatively open subset of $\partial\Omega$. Without lost of generality we  assume that there exists $0<\epsilon<1$ such that for all $\omega\in\{y\in\mathbb S^{n-1}:|y-\omega_0|\leq\epsilon\} $ we have $\partial\Omega_{-,\epsilon,-\omega}\subset F'$.
The goal of this section is to use the Carleman estimate \eqref{c1b} in order to build  solutions $u\in H_{\Box}(Q)$ to 
\begin{equation}
\label{(5.1)}
\left\{
\begin{array}{l}
(\partial_t^2-\Delta_x +q(t,x))u=0\ \ \textrm{in }  Q,
\\
u_{\vert t=0}=0,
\\
u=0,\quad \ \textrm{on } \Sigma_{+,\epsilon/2,-\omega},
\end{array}
\right.
\end{equation}
 of the form
\bel{CGO1a}
u(t,x)=e^{\lambda (t+\omega\cdot x)}\left( 1+z(t,x) \right),\quad (t,x)\in Q.
\ee
Here $\omega\in\{y\in\mathbb S^{n-1}:|y-\omega_0|\leq\epsilon\} $,  $z \in e^{-\lambda (t+\omega\cdot x)}H_{\Box}(Q)$ fulfills: $z(0,x)=-1$ , $x\in\Omega$, $z=-1$ on $\Sigma_{+,\epsilon/2,-\omega}$ and

\bel{CGO1b}
\| z \|_{L^2(Q)}\leq C\lambda^{-\frac{1}{2}}
\ee
with $C$ depending on $F'$, $\Omega$, $T$ and any $M\geq\norm{q}_{L^\infty(Q)}$. Since $\Sigma\setminus F\subset\Sigma\setminus \Sigma_{-,\epsilon,-\omega}=\Sigma_{+,\epsilon,-\omega}$ and since $\Sigma_{+,\epsilon/2,-\omega}$ is a neighborhood of $\Sigma_{+,\epsilon,-\omega}$ in $\Sigma$,  it is clear that condition \eqref{(5.1)} implies $(\tau_{0,1}u,\tau_{0,3}u)\in\mathcal H_F(\partial Q)$  (recall that for $v\in\mathcal C^\infty(\overline{Q})$, $\tau_{0,1}v=v_{|\Sigma}$, $\tau_{0,3}v=\partial_tv_{|t=0}$).

The main result of this section can be stated as follows.

\begin{Thm}\label{tt1} Let $q\in L^\infty(Q)$, $\omega\in\{y\in\mathbb S^{n-1}:|y-\omega_0|\leq\epsilon\} $. For all $\lambda\geq \lambda_1$, with $\lambda_1$ the constant of Theorem \ref{c1},  there exists a solution $u\in H_{\Box}(Q)$ of \eqref{(5.1)} of the form \eqref{CGO1a} with $z$ satisfying \eqref{CGO1b}. \end{Thm}

In order to prove existence of such solutions of \eqref{(5.1)} we need some preliminary tools and an intermediate result.
\subsection{Weighted spaces}

In this subsection we give the definition of some weighted spaces. We set $s\in\R$, we fix $\omega\in\{y\in\mathbb S^{n-1}:|y-\omega_0|\leq\epsilon\} $ and we denote by $\gamma$ the function defined on $\partial\Omega$ by
\[\gamma(x)=\abs{\omega\cdot\nu(x)},\quad x\in\partial\Omega.\]
We introduce  the spaces $L_s(Q)$, $L_s(\Omega)$,  and for all non negative measurable function $h$ on $\partial\Omega$ the spaces $L_{s,h,\pm}$  defined respectively by
\[L_s(Q)=e^{-s(t+\omega\cdot x)}L^2(Q),\quad L_s(\Omega)=e^{-s\omega\cdot x}L^2(\Omega),\quad L_{s,h,\pm}=\{f:\ e^{s(t+\omega\cdot x)}h^{{1\over2}}(x)f\in L^2(\Sigma_{\pm,\omega})\}\]
with the associated norm
\[\norm{u}_s=\left(\int_Qe^{2s(t+\omega\cdot x)}\abs{u}^2dxdt\right)^{\frac{1}{2}},\quad u\in L_s(Q),\]
\[\norm{u}_{s,0}=\left(\int_\Omega e^{2s\omega\cdot x}\abs{u}^2dx\right)^{\frac{1}{2}},\quad u\in L_s(\Omega),\]
\[\norm{u}_{s,h,\pm}=\left(\int_{\Sigma_{\pm,\omega}} e^{2s(t+\omega\cdot x)}h(x)\abs{u}^2d\sigma(x)dt\right)^{\frac{1}{2}},\quad u\in L_{s,h,\pm}.\]

\subsection{Intermediate result}

We set  the space
\[\mathcal D=\{v\in\mathcal C^2(\overline{Q}): \ v_{\vert \Sigma}=0,\ v_{\vert t=T}=\partial_tv_{\vert t=T}=v_{\vert t=0}=0\}\]
and, in view of Theorem \ref{c1}, applying the Carleman estimate \eqref{c1b} to any $f\in \mathcal D$ we obtain
\begin{equation}\label{caca}\lambda\norm{f}_\lambda+\lambda^{\frac{1}{2}}\norm{\partial_tf_{\vert t=0}}_{\lambda,0}+\lambda^{\frac{1}{2}}\norm{\partial_\nu f}_{\lambda,\gamma,-}\leq C(\norm{(\partial_t^2-\Delta_x+q)f}_\lambda+\norm{\partial_\nu f}_{\lambda,\lambda\gamma,+}),\quad \lambda\geq \lambda_1.\end{equation}
We introduce also the space
\[\mathcal M=\{((\partial_t^2-\Delta_x+q)v,\partial_\nu v_{\vert\Sigma_{+,\omega}}):\ v\in\mathcal D\}\]
and  think of $\mathcal M$ as a subspace of $L_\lambda(Q)\times L_{\lambda,\lambda\gamma,+}$. We consider the following intermediate result.

\begin{lem}\label{ppp1} Given $\lambda\geq \lambda_1$, with $\lambda_1$ the constant of Theorem \ref{c1}, and
\[v\in L_{-\lambda}(Q),\quad v_-\in L_{-\lambda,\gamma^{-1},-},\quad v_0\in L_{-\lambda}(\Omega),\]
there exists  $u\in L_{-\lambda}(Q)$ such that:\\
1) $(\partial_t^2-\Delta_x+q)u=v$,\\
2) $u_{\vert \Sigma_{-,\omega}}=v_-,\ u_{|t=0}=v_0$,\\
3) $\norm{u}_{-\lambda}\leq  C\left(\lambda^{-1}\norm{v}_{-\lambda}+\lambda^{-\frac{1}{2}}\norm{v_-}_{-\lambda,\gamma^{-1},-}+\lambda^{-\frac{1}{2}}\norm{v_0}_{-\lambda,0}\right)$ with $C$ depending on $\Omega$, $T$,\\
 $M\geq\norm{q}_{L^\infty(Q)}$.
\end{lem}
\begin{proof} In view of \eqref{caca}, we can define the linear function $S$ on $\mathcal M$ by
\[S[((\Box+q) f,\partial_\nu f_{\vert\Sigma_{+,\omega}})]=\left\langle f,v \right\rangle_{L^2(Q)}-\left\langle \partial_\nu f,v_- \right\rangle_{L^2(\Sigma_{-,\omega})}+\left\langle \partial_tf_{\vert t=0},v_0 \right\rangle_{L^2(\Omega)},\quad f\in\mathcal D.\]
Then, using \eqref{caca}, for all $f\in\mathcal D$, we obtain
 \[\begin{array}{l}\abs{S[((\Box+q) f,\partial_\nu f_{\vert\Sigma_{+,\omega}})]}\\
\leq \norm{f}_{\lambda}\norm{v}_{-\lambda}+\norm{\partial_\nu f}_{\lambda,\gamma,-}\norm{v_-}_{-\lambda,\gamma^{-1},-}+\norm{\partial_tf_{\vert t=0}}_{\lambda,0}\norm{v_0}_{-\lambda,0}\\
\leq \lambda^{-1}\norm{v}_{-\lambda}\left(\lambda\norm{f}_{\lambda}\right)+\lambda^{-\frac{1}{2}}\norm{v_-}_{-\lambda,\gamma^{-1},-}\left(\lambda^{\frac{1}{2}}\norm{\partial_\nu f}_{\lambda,\gamma,-}\right)+ \lambda^{-\frac{1}{2}}\norm{v_0}_{-\lambda,0}\left(\lambda^{\frac{1}{2}}\norm{\partial_tf_{\vert t=0}}_{\lambda,0}\right)\\
\leq C\left( \lambda^{-1}\norm{v}_{-\lambda}+\lambda^{-\frac{1}{2}}\norm{v_-}_{-\lambda,\gamma^{-1},-}+\lambda^{-\frac{1}{2}}\norm{v_0}_{-\lambda,0}\right)\left( \norm{(\Box+q) f}_{\lambda}+\norm{\partial_\nu f}_{\lambda,\lambda\gamma,+}\right)\\
\leq 2C\left( \lambda^{-1}\norm{v}_{-\lambda}+\lambda^{-\frac{1}{2}}\norm{v_-}_{-\lambda,\gamma^{-1},-}+\lambda^{-\frac{1}{2}}\norm{v_0}_{-\lambda,0}\right) \norm{((\Box+q) f,\partial_\nu f_{\vert\Sigma_{+,\omega}})}_{L_\lambda(Q)\times L_{\lambda,\lambda\gamma,+}}\end{array}\]
with $C$ the constant of \eqref{caca}.  Applying the Hahn Banach theorem we deduce that $S$ can be extended to a continuous  linear form, also denoted by $S$, on 
$L_\lambda(Q)\times L_{\lambda,\lambda\gamma,+}$ satisfying
\bel{ppp1a}\norm{S}\leq C\left( \lambda^{-1}\norm{v}_{-\lambda}+\lambda^{-\frac{1}{2}}\norm{v_-}_{-\lambda,\gamma^{-1},-}+\lambda^{-\frac{1}{2}}\norm{v_0}_{-\lambda,0}\right).\ee
Thus, there exists
\[(u,u_+)\in L_{-\lambda}(Q)\times L_{-\lambda,(\lambda\gamma)^{-1},+}\]
such that for all $f\in\mathcal D$ we have
\[S[((\Box+q) f,\partial_\nu f_{\vert\Sigma_{+,\omega}})]=\left\langle (\Box+q) f,u \right\rangle_{L^2(Q)}-\left\langle\partial_\nu f,u_+ \right\rangle_{L^2(\Sigma_{+,\omega})}.\]
Therefore, for all $f\in\mathcal D$ we have
\bel{ppp1b}\begin{array}{l} \left\langle (\Box+q) f,u \right\rangle_{L^2(Q)}-\left\langle\partial_\nu f,u_+ \right\rangle_{L^2(\Sigma_{+,\omega})}\\
=\left\langle f,v \right\rangle_{L^2(Q)}-\left\langle \partial_\nu f,v_- \right\rangle_{L^2(\Sigma_{-,\omega})}+\left\langle \partial_tf_{\vert t=0},v_0 \right\rangle_{L^2(\Omega)}.\end{array}\ee

Note first that, since $L_{\pm \lambda}(Q)$ embedded continuously into $L^2(Q)$, we have $u\in L^2(Q)$.
Therefore, taking $f\in\mathcal C^\infty_0(Q)$ shows 1). For condition 2), using the fact that $L_{\pm \lambda}(Q)$ embedded continuously into $L^2(Q)$ we deduce that $u\in H_\Box(Q)$. Thus, we can define the trace $u_{\vert\Sigma}$, $u_{\vert t=0}$ and  allowing $f\in\mathcal D$ to be arbitrary shows that $u_{\vert\Sigma_{-,\omega}} =v_-$, $u_{\vert t=0}=v_0$ and $u_{\vert\Sigma_{+,\omega}}=-u_+$. Finally, condition 3) follows from the fact that  
\[\norm{u}_{-\lambda}\leq \norm{S}\leq C\left( \lambda^{-1}\norm{v}_{-\lambda}+\lambda^{-\frac{1}{2}}\norm{v_-}_{-\lambda,\gamma^{-1},-}+\lambda^{-\frac{1}{2}}\norm{v_0}_{-\lambda,0}\right).\]\end{proof}

Armed with this lemma we are now in position to prove Theorem \ref{tt1}. 
\subsection{Proof of Theorem \ref{tt1}}

Note first that $z$ must satisfy
\begin{equation}
\label{w1}
\left\{
\begin{array}{l}z\in L^2(Q) \\
(\partial_t^2-\Delta_x+q) (e^{\lambda(t+\omega\cdot x)}z)=-qe^{\lambda(t+\omega\cdot x)}\ \ \textrm{in }Q
\\
z(0,x)=-1, \quad   x\in\Omega,
\\
z=-1\quad \textrm{on }\Sigma_{+,\epsilon/2,-\omega}.
\end{array}
\right.\end{equation}
Let $\psi\in\mathcal C^\infty_0(\R^n)$ be such that   supp$\psi\cap\partial\Omega\subset \{x\in\partial\Omega:\ \omega\cdot\nu(x)<-\epsilon/3\}$ and $\psi=1$ on $\{x\in\partial\Omega:\ \omega\cdot\nu(x)<-\epsilon/2\}=\partial\Omega_{+,\epsilon/2,-\omega}$. Choose $v_-(t,x)=-e^{\lambda(t+\omega\cdot x)}\psi(x)$, $(t,x)\in\Sigma_{-,\omega}$.  Since $v_-(t,x)=0$ for $t\in(0,T)$, $x\in \{x\in\partial\Omega:\ \omega\cdot\nu(x)\geq-\epsilon/3\}$  we have
$v_-\in L_{-\lambda,\gamma^{-1},-}$. Fix also $v(t,x)=-qe^{\lambda(t+\omega\cdot x)}$ and $v_0(x)=-e^{\lambda\omega\cdot x}$, $(t,x)\in Q$. From Lemma \ref{ppp1}, we deduce that there exists $w\in H_\Box(Q)$ such that
\[
\left\{
\begin{array}{ll}
(\partial_t^2-\Delta_x+q) w=v(t,x)=-qe^{\lambda(t+\omega\cdot x)}&  \mbox{in}\; Q,
\\
w(0,x)=v_0(x)=-e^{\lambda x\cdot\omega}, &  x\in\Omega,
\\
w(t,x)=v_-(t,x)=-e^{\lambda(t+\omega\cdot x)}\psi(x),& (t,x)\in\Sigma_{-,\omega}.
\end{array}
\right.\]
Then, for $z=e^{-\lambda(t+\omega\cdot x)} w$ condition \eqref{w1} will be fulfilled. Moreover,  condition 3) of Lemma \ref{ppp1} implies
\[\begin{aligned}\norm{z}_{L^2(Q)}=\norm{w}_{-\lambda}&\leq C\left(\lambda^{-1}\norm{v}_{-\lambda}+\lambda^{-\frac{1}{2}}\norm{v_-}_{-\lambda,\gamma^{-1},-}+\lambda^{-\frac{1}{2}}\norm{v_0}_{-\lambda,0}\right)\\
\ &\leq \left(\lambda^{-1}\norm{q}_{L^2(Q)}+\lambda^{-\frac{1}{2}}\norm{\psi\gamma^{-1/2}}_{L^2(\Sigma_{-,\omega})}+\lambda^{-\frac{1}{2}}\norm{1}_{L^2(\Omega)}\right)
\ &\leq C\lambda^{-{1\over2}}\end{aligned}\]
with $C$ depending only on $\Omega$, $T$ and $\norm{q}_{L^\infty(Q)}$. Therefore,  estimate \eqref{CGO1b} holds.  Using the fact that $e^{\lambda(t+\omega\cdot x)} z=w\in H_{\Box}(Q)$, we deduce that  $u$ defined by \eqref{CGO1a} is lying in $H_{\Box}(Q)$ and is a solution of \eqref{(5.1)}. This completes the proof of Theorem \ref{tt1}.

\section{Uniqueness result}
This section is devoted to the proof of Theorem \ref{thm1}.  From now on we set $q=q_2-q_1$ on $Q$ and  we assume  that $q=0$ on $\R^{1+n}\setminus Q$. Without lost of generality we assume that  for all $\omega\in\{y\in\mathbb S^{n-1}:|y-\omega_0|\leq\epsilon\} $ we have $\partial\Omega_{-,\epsilon,\omega}\subset G'$ with $\epsilon>0$ introduced in the beginning of the previous section. 
Let   $\lambda >\max(\lambda_1,\lambda_0) $ and fix $\omega\in\{y\in\mathbb S^{n-1}:|y-\omega_0|\leq\epsilon\} $. According to Proposition \ref{p2}, we can introduce
\[u_1(t,x)=e^{-\lambda(t+\omega\cdot x)}\left(e^{-i\xi\cdot(t,x)}+w(t,x) \right) ,\ (t,x) \in Q,\]
where $u_1\in H^1(Q)$ satisfies $\partial_t^2u_1-\Delta_xu_1+q_1u_1=0$, $\xi\cdot(1,-\omega)=0$ and  $w$ satisfies \eqref{p2a}. Moreover, in view of Theorem \ref{tt1}, we consider $u_2\in H_{\Box}(Q)$ a solution of \eqref{(5.1)} with $q=q_2$ of the form 
\[u_2(t,x)=e^{\lambda(t+\omega\cdot x)}\left(1+z(t,x) \right),\quad (t,x)\in Q\]
with $z$ satisfying \eqref{CGO1b}, such that supp$\tau_{0,1}u_{2}\subset F$ and $\tau_{0,2}u_2=0$ (we recall that $\tau_{0,j}$, $j=1,2$, are the extensions on $H_{\Box}(Q)$ of the operators defined   by $\tau_{0,1}v=v_{|\Sigma}$ and $\tau_{0,2}v=v_{|t=0}$, $v\in\mathcal C^\infty(\overline{Q})$) .
In view of Proposition \ref{p6}, there exists a unique solution $w_1\in H_\Box (Q)$ of
 \bel{eq3}
\left\{
\begin{array}{ll}
 \partial_t^2w_1-\Delta_xw_1 +q_1w_1=0 &\mbox{in}\ Q,
\\

\tau_{0}w_1=\tau_{0}u_2. &

\end{array}
\right.
\ee
Then, $u=w_1-u_2$ solves
  \bel{eq4}
\left\{\begin{array}{ll}
 \partial_t^2u-\Delta_xu +q_1u=(q_2-q_1)u_2 &\mbox{in}\ Q,
\\
u(0,x)=\partial_tu(0,x)=0 & \mathrm{on}\  \Omega,\\

u=0 &\mbox{on}\ \Sigma
\end{array}\right.
\ee
and since $(q_2-q_1)u_2\in L^2(Q)$, in view of  \cite[Theorem A.2]{BCY} (see also   \cite[Theorem 2.1]{LLT} for $q=0$), we deduce that $u\in \mathcal C^1([0,T];L^2(\Omega))\cap \mathcal C([0,T];H^1_0(\Omega))\cap H_{\Box}(Q)\subset H^1(Q)\cap H_{\Box}(Q)$ with $\partial_\nu u\in L^2(\Sigma)$. Using the fact that $u_1\in H^1(Q)\cap H_\Box(Q)$, we deduce that $(\partial_tu_1,-\nabla_xu_1)\in H_{\textrm{div}}(Q)=\{F\in L^2(Q;\mathbb C^{n+1}):\ \textrm{div}_{(t,x)}F\in L^2(Q)\}$. Therefore, in view of \cite[Lemma 2.2]{Ka}, we can apply the Green formula to get
\[\int_Q u(\Box u_1)dtdx =-\int_Q(\partial_t u\partial_t u_1-\nabla_x u\cdot\nabla_xu_1)dtdx+\left\langle(\partial_tu_1,-\nabla_xu_1)\cdot \textbf{n}, u\right\rangle_{H^{-{1\over2}}(\partial Q),H^{{1\over2}}(\partial Q)}\]
with $\textbf{n}$ the outward unit normal vector to $Q$. In the same way, we find
\[\int_Q u_1(\Box u)dtdx =-\int_Q(\partial_t u\partial_t u_1-\nabla_x u\cdot\nabla_xu_1)dtdx+\left\langle(\partial_tu,-\nabla_xu)\cdot \textbf{n}, u_1\right\rangle_{H^{-{1\over2}}(\partial Q),H^{{1\over2}}(\partial Q)}.\]
From these two formulas we deduce that
\[\begin{aligned}\int_Q(q_2-q_1)u_2u_1dtdx&=\int_Q u_1(\Box u+q_1 u)dtdx-\int_Q u(\Box u_1+q_1u_1)dtdx\\
&=\left\langle(\partial_tu,-\nabla_xu)\cdot \textbf{n}, u_1\right\rangle_{H^{-{1\over2}}(\partial Q),H^{{1\over2}}(\partial Q)}-\left\langle(\partial_tu_1,-\nabla_xu_1)\cdot \textbf{n}, u\right\rangle_{H^{-{1\over2}}(\partial Q),H^{{1\over2}}(\partial Q)}.\end{aligned}\]
On the other hand we have $u_{|t=0}=\partial_tu_{|t=0}=u_{|\Sigma}=0$ and condition \eqref{thm1a} implies that $u_{|t=T}=\partial_\nu u_{|G}=0$. Combining this with the fact that $u\in \mathcal C^1([0,T];L^2(\Omega))$ and $\partial_\nu u\in L^2(\Sigma)$, we obtain
\begin{equation}\label{t3a} \int_Qqu_2u_1dtdx=-\int_{\Sigma\setminus G}\partial_\nu uu_1d\sigma(x)dt+\int_\Omega \partial_tu(T,x)u_1(T,x)dx.\end{equation}
Applying  the Cauchy-Schwarz inequality to the first expression on the right hand side of this formula, we get
\[\begin{aligned}\abs{\int_{\Sigma\setminus G}\partial_\nu uu_1d\sigma(x)dt}&\leq\int_{{\Sigma}_{+,\epsilon,\omega}}\abs{\partial_\nu ue^{-\lambda(t+\omega\cdot x)}(1+w)}dt d\sigma(x) \\
 \ &\leq C\left(\int_{{\Sigma}_{+,\epsilon,\omega}}\abs{e^{-\lambda(t+\omega\cdot x)}\partial_\nu u}^2d\sigma(x)dt\right)^{\frac{1}{2}}\end{aligned}\]
for some $C$ independent of $\lambda$. Here we have used both \eqref{p2a} and the fact that $(\Sigma\setminus G)\subset {\Sigma}_{+,\epsilon,\omega}$. In the same way, we have
\[\begin{aligned}\abs{\int_\Omega \partial_tu(T,x)u_1(T,x)dx}&\leq\int_{\Omega}\abs{\partial_t u(T,x)e^{-\lambda(T+\omega\cdot x)}(1+w(T,x))}dx \\
 \ &\leq C\left(\int_{\Omega}\abs{e^{-\lambda(T+\omega\cdot x)}\partial_t u(T,x)}^2dx\right)^{\frac{1}{2}}.\end{aligned}\]
Combining these estimates with the Carleman estimate \eqref{c1a}, the fact that $u_{|t=T}=\partial_\nu u_{|\Sigma_{-,\omega}}=0$, ${\partial\Omega}_{+,\epsilon,\omega}\subset {\partial\Omega}_{+,\omega}$, we find
\[\begin{array}{l}\abs{\int_Q(q_2-q_1)u_2u_1dtdx}^2\\

\leq 2C\left(\int_{{\Sigma}_{+,\epsilon,\omega}}\abs{e^{-\lambda(t+\omega\cdot x)}\partial_\nu u}^2d\sigma(x)dt+\int_\Omega \abs{e^{-\lambda(T+\omega\cdot x)}\partial_tu(T,x)}^2dx\right)\\
\ \\

\leq 2\epsilon^{-1}C\left(\int_{{\Sigma}_{+,\omega}}\abs{e^{-\lambda(t+\omega\cdot x)}\partial_\nu u}^2\omega\cdot \nu(x)d\sigma(x)dt+\int_\Omega \abs{e^{-\lambda(T+\omega\cdot x)}\partial_tu(T,x)}^2dx\right)\\
\ \\
\leq {\epsilon^{-1}C\over\lambda}\left(\int_Q\abs{ e^{-\lambda(t+\omega\cdot x)}(\partial_t^2-\Delta_x +q_1)u}^2dxdt\right)\\
\ \\
\leq {\epsilon^{-1}C\over\lambda}\left(\int_Q\abs{ e^{-\lambda(t+\omega\cdot x)}qu_2}^2dxdt\right)={\epsilon^{-1}C\over\lambda}\left(\int_Q\abs{ q}^2(1+\abs{z})^2dxdt\right).\end{array}\]
Here $C>0$ stands for some generic constant independent of $\lambda$.
It follows that
\begin{equation}\label{t3cc}\limsup_{\lambda\to+\infty}\int_Qqu_2u_1dtdx=0.\end{equation}
On the other hand, we have
\[\int_{Q}qu_1u_2d xdt=\int_{\R^{1+n}}q(t,x)e^{-i\xi\cdot(t,x)}dxdt+ \int_{Q}Z(t,x) dxdt\]
with $ Z(t,x)=q(t,x)(z(t,x)e^{-i\xi\cdot(t,x)}+w(t,x)+z(t,x)w(t,x))$. Then, in view of \eqref{p2a}  and \eqref{CGO1b}, an application of the Cauchy-Schwarz inequality yields
\[\abs{\int_{Q}Z(t,x)  dxdt}\leq C\lambda^{-\frac{1}{2}}\]
with $C$ independent of $\lambda$. Combining this with \eqref{t3cc}, we deduce that for all $\omega\in\{y\in\mathbb S^{n-1}:|y-\omega_0|\leq\epsilon\} $ and  all $\xi\in\R^{1+n}$ orthogonal to $(1,-\omega)$, the Fourier transform $\mathcal F(q)$ of $q$ satisfies
\[\mathcal F(q)(\xi)=(2\pi)^{-{n+1\over2}}\int_{\R^{1+n}}q(t,x)e^{-i\xi\cdot(t,x)}dxdt=0.\]
On the other hand, since $q\in L^\infty(Q)$ is compactly supported,  $\mathcal F(q)$ is analytic and it follows that $q=0$ and $q_1=q_2$. This completes the proof of Theorem \ref{thm1}.

\section*{Appendix}
In this appendix we prove  that the space $\mathcal C^\infty(\overline{Q})$ is dense in $H_\Box(Q)$ in some appropriate sense and we show that the  maps $\tau_0$ and $\tau_1$ can be extended continuously on these spaces. Without lost of generality we consider only these spaces for real valued functions. The results of this section are well known, nevertheless we prove them for sake of completeness.

\subsection*{Density result in $H_\Box(Q)$}

Let us first recall the definition of $K_\Box(Q)$:
\[K_\Box(Q)=\{u\in H^{-1}(0,T;L^2(\Omega)):\ \Box u=(\partial_t^2-\Delta_x) u\in L^2(Q)\}\]
with the norm
\[\norm{u}_{K_\Box(Q)}^2=\norm{u}_{H^{-1}(0,T;L^2(\Omega))}^2+\norm{\Box u}_{L^2(Q)}^2.\]
The goal of this subsection is to prove the following.

\begin{Thm}\label{density} $H_\Box(Q)$ embedded continuously into the closure of $\mathcal C^\infty(\overline{Q})$ with respect to $K_\Box(Q)$.\end{Thm}
\begin{proof} Let $N$ be a continuous linear form on $K_\Box(Q)$ satisfying
\bel{d1}Nf=0,\quad f\in \mathcal C^\infty(\overline{Q}).\ee
In order to show the required density result we will prove that this condition implies that\\
 $N_{\vert H_\Box(Q)}=0$.

 By considering the application $u\mapsto (u,\Box u)$ we can identify $K_\Box(Q)$ to a subspace of $H^{-1}(0,T;L^2(\Omega))\times L^2(Q)$. Then, applying the Hahn Banach theorem we deduce that $N$ can be extended to a continuous linear form on $H^{-1}(0,T;L^2(\Omega))\times L^2(Q)$. Therefore, there exist $h_1\in H^1_0(0,T;L^2(\Omega))$, $h_2\in L^2(Q)$ such that
\[N(u)=\left\langle u,h_1\right\rangle_{H^{-1}(0,T;L^2(\Omega)),H^1_0(0,T;L^2(\Omega))}+\left\langle \Box u,h_2\right\rangle_{L^2(Q)},\quad u\in K_\Box(Q).\]
Now let $\mathcal O\subset\R^n$ be a bounded $\mathcal C^\infty$ domain  such that $\overline{\Omega}\subset\mathcal O $ and fix $Q_\epsilon=(-\epsilon,T+\epsilon)\times\mathcal O$ with $\epsilon>0$. Let $\tilde{h}_j$ be the extension of $h_j$ on $\R^{1+n}$ by $0$ outside of $Q$ for $j=1,2$. In view of \eqref{d1} we have
\[\left\langle f,\tilde{h}_1\right\rangle_{L^2(Q_\epsilon)}+\left\langle (\partial_t^2-\Delta_x) f,\tilde{h}_2\right\rangle_{L^2(Q_\epsilon)}=N(f_{\vert Q})=0,\quad f\in\mathcal C^\infty_0(Q_\epsilon).\]
Thus, in the sense of distribution we have 
\[\Box \tilde{h}_2=-\tilde{h}_1\quad \textrm{on }Q_\epsilon.\]
Moreover, since $\tilde{h}_2=0$ on $\R^{1+n}\setminus \overline{Q}\supset \partial Q_\epsilon$, we deduce that $\tilde{h}_2$ solves
\[
\left\{ \begin{array}{rcll} \partial_t^2\tilde{h}_2-\Delta_x\tilde{h}_2& = & -\tilde{h}_1 & \textrm{in }  Q_\epsilon ,\\ 
\tilde{h}_2(-\epsilon,x)=\partial_t\tilde{h}_2(-\epsilon,x)&=&0,&x\in \mathcal O,\\
\tilde{h}_2(t,x)& = & 0,& (t,x)\in (-\epsilon,T+\epsilon)\times\partial\mathcal O.\end{array}\right.
\]
But, since $h_1\in H^1_0(0,T;L^2(\Omega))$, we have $\tilde{h}_1\in H^1_0(-\epsilon,T+\epsilon;L^2(\mathcal O))$ and we deduce from  \cite[Theorem 2.1, Chapter 5]{LM2}
 that this IBVP admits a unique solution lying in $H^2(Q_\epsilon)$. Therefore, $\tilde{h}_2\in H^2(Q_\epsilon)$. Combining this with the fact that $\tilde{h}_2=0$ on $Q_\epsilon\setminus Q$, we deduce that $h_2\in  H_0^2(Q)$, with $H_0^2(Q)$ the closure of $\mathcal C^\infty_0(Q)$ in $H^2(Q)$, and that $\Box h_2=-h_1$ on $Q$. Thus, for every $u\in H_\Box(Q)$ we have
\[\left\langle \Box u,h_2\right\rangle_{L^2(Q)}=\left\langle \Box u,h_2\right\rangle_{H^{-2}(Q),H_0^2(Q)}=\left\langle u,\Box h_2\right\rangle_{L^2(Q)}=-\left\langle  u,h_1\right\rangle_{L^2(Q)}.\]
Here we use the fact that $H_\Box(Q)\subset L^2(Q)$. Then, it follows that 
\[N(u)=\left\langle u,h_1\right\rangle_{L^2(Q)}-\left\langle u,h_1\right\rangle_{L^2(Q)}=0,\quad u\in H_\Box(Q).\]
From this last result we deduce that $H_\Box(Q)$ is contained into the closure of $\mathcal C^\infty(\overline{Q})$ with respect to $K_\Box(Q)$. Combining this with the fact that
$H_\Box(Q)$ embedded continuously into $K_\Box(Q)$, we deduce the required result.
\end{proof}

\subsection*{Trace operator in $H_\Box(Q)$}
In this subsection we extend the trace maps $\tau_0$ and $\tau_1$ into $H_\Box(Q)$ by duality in the following way.
\begin{prop}\label{trace}The maps
\[\tau_0w=(\tau_{0,1}w,\tau_{0,2}w,\tau_{0,3}w)=(w_{\vert\Sigma},w_{\vert t=0}, \partial_tw_{\vert t=0}), \quad w\in \mathcal C^\infty(\overline{Q}),\]
\[\tau_1w=(\tau_{1,1}w,\tau_{1,2}w,\tau_{1,3}w)=(\partial_\nu w_{\vert\Sigma},w_{\vert t=T}, \partial_tw_{\vert t=T}), \quad w\in \mathcal C^\infty(\overline{Q}),\]
can be extended continuously to $\tau_0:H_\Box(Q)\rightarrow H^{-3}(0,T; H^{-\frac{1}{2}}(\partial\Omega))\times H^{-2}(\Omega)\times H^{-4}(\Omega)$,\\ $\tau_1:H_\Box(Q)\rightarrow H^{-3}(0,T; H^{-\frac{3}{2}}(\partial\Omega))\times H^{-2}(\Omega)\times H^{-4}(\Omega)$.\end{prop}
\begin{proof}  It is well known that the trace maps $$u\mapsto (u_{\vert\partial\Omega}, \partial_\nu u_{\vert\partial\Omega})$$ can be extended continuously to a bounded operator from $H^2(\Omega)$ to $H^{\frac{3}{2}}(\partial\Omega)\times H^{\frac{1}{2}}(\partial\Omega)$ which is onto. Therefore, there exists a bounded operator $R:H^{\frac{3}{2}}(\partial\Omega)\times H^{\frac{1}{2}}(\partial\Omega)\rightarrow H^2(\Omega)$ such that 
\[R(h_1,h_2)_{\vert\partial\Omega}=h_1,\quad \partial_\nu R(h_1,h_2)_{\vert\partial\Omega}=h_2,\quad (h_1,h_2)\in H^{\frac{3}{2}}(\partial\Omega)\times H^{\frac{1}{2}}(\partial\Omega).\]
Fix $g\in H^3_0(0,T;H^{\frac{1}{2}}(\partial\Omega))$ and choose $G(t,.)=R(0,g(t,.))$. One can check that $G\in H^3_0(0,T;H^2(\Omega))$ and
\bel{tr1}\norm{G}_{H^3(0,T;H^2(\Omega))}\leq \norm{R}\norm{g}_{H^3(0,T;H^{\frac{1}{2}}(\partial\Omega))}.\ee
Applying twice the Green formula we obtain
 \[\int_{\Sigma} vgd\sigma(x)dt=\int_Q \Box v Gdxdt- \int_Q  v\Box Gdxdt,\quad v\in \mathcal C^\infty(\overline{Q}).\]
But $\Box G\in H^1_0(0,T;H^2(\Omega))$, and we have
\[\left\langle  \tau_{0,1}v,g\right\rangle_{H^{-3}(0,T;H^{-\frac{1}{2}}(\partial\Omega)),H^3_0(0,T;H^{\frac{1}{2}}(\partial\Omega))}=\left\langle \Box v,G\right\rangle_{L^2(Q)}-\left\langle v,\Box G\right\rangle_{H^{-1}(0,T;L^2(\Omega)),H^1_0(0,T;L^2(\Omega))}.\]
Then, using \eqref{tr1} and the Cauchy Schwarz inequality, for all $v\in \mathcal C^\infty(\overline{Q})$, we obtain
 \[\begin{aligned}\abs{\left\langle  \tau_{0,1}v,g\right\rangle}&\leq \norm{\Box v}_{L^2(Q)}\norm{G}_{L^2(Q)}+\norm{v}_{H^{-1}(0,T;L^2(\Omega))}\norm{\Box G}_{H^1_0(0,T;L^2(\Omega))}\\
\ &\leq C\norm{v}_{K_\Box(Q)}\norm{g}_{H^3(0,T;H^{\frac{1}{2}}(\partial\Omega))}\end{aligned}\]
which, combined with the density result of Theorem \ref{density}, implies that $\tau_{0,1}:\ v\mapsto v_{\vert\Sigma}$ extend continuously to a bounded operator from $H_\Box(Q)$ to 
$H^{-3}(0,T;H^{-\frac{1}{2}}(\partial\Omega))$. In a same way we prove that 
\[\tau_{1,1}v=\partial_\nu v_{\vert\Sigma},\quad v\in \mathcal C^\infty(\overline{Q})\]
extend continuously to a bounded operator from $H_\Box(Q)$ to 
$H^{-3}(0,T;H^{-\frac{3}{2}}(\partial\Omega))$. 

Now let us consider the operators $\tau_{i,j}$, $i=0,1$, $j=2,3$. We start with
\[\tau_{0,2}:v\longmapsto v_{\vert t=0},\quad v\in\mathcal C^\infty(\overline{Q}).\]
Let $h\in H^2_0(\Omega)$ and fix $H(t,x)=t\psi(t)h(x)$ with
$\psi\in\mathcal C^\infty_0(-T,\frac{T}{2})$ satisfying $0\leq \psi\leq 1$ and $\psi=1$ on $[-\frac{T}{3},\frac{T}{3}]$. Then, using the fact that $\psi=1$ on a neighborhood of $t=0$, we deduce that
\[H_{\vert\Sigma}=\partial_\nu H_{\vert\Sigma}= H_{\vert t=0}=\Box H_{\vert t=0}=\Box H_{\vert t=T}=0,\quad \partial_tH_{\vert t=0}=h.\]
Therefore, $\Box H\in H^1_0(0,T; L^2(\Omega))$ and repeating the above arguments, for all $v\in\mathcal C^\infty(\overline{Q})$, we obtain the representation 
 \[\left\langle  \tau_{0,2}v,h\right\rangle_{H^{-2}(\Omega),H^2_0(\Omega)}=\left\langle v,\Box H\right\rangle_{H^{-1}(0,T;L^2(\Omega)),H^1_0(0,T;L^2(\Omega))}-\left\langle H,\Box v\right\rangle_{L^2(Q)}.\]
Then, we prove by density that $\tau_{0,2}$ extends continuously to $\tau_{0,2}:H_\Box(Q)\longrightarrow H^{-2}(\Omega)$.

For
\[\tau_{0,3}:v\longmapsto \partial_t v_{\vert t=0},\quad v\in\mathcal C^\infty(\overline{Q}),\]
let $\phi\in H^4_0(\Omega)$ and fix
\[\Phi(t,x)=\psi(t)\phi(x)+\frac{\psi(t)t^2\Delta_x\phi(x)}{2}.\]
 Then, $\Phi$ satisfies
\[\Phi_{\vert\Sigma}=\partial_\nu \Phi_{\vert\Sigma}=\partial_t \Phi_{\vert t=0}=0,\quad \Phi_{\vert t=0}=\phi.\]
Moreover, we have $\Box \Phi\in H^1(0,T;L^2(\Omega))$ with
\[(\partial_t^2-\Delta_x) \Phi_{\vert t=0}=-\Delta_x\phi+\Delta_x\phi=0,\quad (\partial_t^2-\Delta_x) \Phi_{\vert t=T}=0\]
and it follows that $\Box \Phi\in H_0^1(0,T;L^2(\Omega))$. Therefore, repeating the above arguments we obtain the representation 
 \[\left\langle  \tau_{0,3}v,\phi\right\rangle_{H^{-4}(\Omega),H^4_0(\Omega)}=\left\langle \Box v,\Phi\right\rangle_{L^2(Q)}-\left\langle v,\Box \Phi\right\rangle_{H^{-1}(0,T;L^2(\Omega)),H^1_0(0,T;L^2(\Omega))}\]
and we deduce that $\tau_{0,3}$ extends continuously to $\tau_{0,3}:H_\Box(Q)\longrightarrow H^{-4}(\Omega)$.
In a same way, one can check that 
\[\tau_{1,2}v=v_{\vert t=T},\ \tau_{1,3}v= \partial_t v_{\vert t=T},\quad v\in\mathcal C^\infty(\overline{Q})\]
 extend  continuously to $\tau_{1,2}:H_\Box(Q)\longrightarrow H^{-2}(\Omega)$ and $\tau_{1,3}:H_\Box(Q)\longrightarrow H^{-4}(\Omega)$.
\end{proof}

\end{document}